\documentclass[10pt,oneside,reqno]{amsart}
\usepackage[a4paper]{geometry}
\usepackage{tikz}
	\usetikzlibrary{trees,shapes.geometric,arrows}
\usepackage[english]{babel}
\usepackage{color}
	\definecolor{darkgreen}{cmyk}{0.5,0.5,1,0.4}
	\definecolor{grayA}{gray}{0.85}
	\definecolor{grayB}{gray}{0.65}

\usepackage{multicol}
\usepackage{verbatim}
\usepackage{float} 
\usepackage{amsthm,mathtools}
\usepackage{hyperref} 
	\hypersetup{colorlinks=true, linkcolor=darkgreen, citecolor=darkgreen} 
\usepackage[hypcap]{caption}
\usepackage{paralist}
\usepackage{amsfonts, amssymb}
\usepackage{thmtools}
	\declaretheoremstyle[bodyfont=\normalfont]{noncursive}
	\declaretheorem[name=Theorem,numberwithin=section]{theorem}
	\declaretheorem[sibling=theorem,style=noncursive]{definition}
	\declaretheorem[sibling=theorem]{corollary} 
	\declaretheorem[sibling=theorem]{proposition}
	\declaretheorem[sibling=theorem]{lemma}
	\declaretheorem[sibling=theorem,style=remark]{remark}
	
	\declaretheorem[sibling=theorem,style=definition]{example}

	\numberwithin{equation}{section} 

	\newcommand{\N}{\mathbb N}
	\newcommand{\Z}{\mathbb Z}
	
	\newcommand{\R}{\mathbb R}
	\newcommand{\C}{\mathbb C}
	\DeclareMathOperator{\UnitSphere}{\mathbb S^1}
	\DeclareMathOperator{\re}{Re}
	\DeclareMathOperator{\im}{Im}
	\DeclareMathOperator{\Aut}{Aut}
	\DeclareMathOperator{\Auto}{Aut_0}
	\DeclareMathOperator{\stab}{stab}

	\DeclareMathOperator{\rank}{rank}
	\DeclareMathOperator{\image}{image}
	\DeclareMathOperator{\id}{id}
	\DeclareMathOperator{\incl}{incl}

	\DeclareMathOperator{\imu}{\mathrm i}

	\DeclareMathOperator{\eps}{\varepsilon}
	\DeclareMathOperator{\NTwo}{\mathcal N}
	\DeclareMathOperator{\FreeNTwo}{\mathfrak N}
	
	\DeclareMathOperator{\FTwo}{\mathcal F}
	\DeclareMathOperator{\FreeFTwo}{\mathfrak F}
	\DeclareMathOperator{\JetFTwo}{\mathfrak J}
	\DeclareMathOperator{\Isotropies}{\mathcal G}

	\newcommand{\HeisenbergOne}[1]{\mathbb H^{#1}}
	\newcommand{\HeisenbergTwo}[2]{\mathbb H_{#1}^{#2}}
	
	\newcommand{\Sphere}[1]{\mathbb S^{#1}}
	\newcommand{\Hyperquadric}[2]{\mathbb S_{#1}^{#2}}

	\newcommand{\MapMainTheoremTwo}[2]{H_{#1}^{#2}}
	\newcommand{\MapOneParameterFamilyTwo}[2]{G_{#1}^{#2}} 
	\newcommand{\MapOneParameterFamilyThree}[3]{G_{#1,#2}^{#3}}

\begin{document}
\title{Topological Aspects of Holomorphic Mappings of Hyperquadrics from $\C^2$ to $\C^3$}
\author{ Michael Reiter}
\thanks{The author was supported by the FWF, projects Y377 and I382, and QNRF, project NPRP 7-511-1-098.}

\pagestyle{plain}

\begin{abstract}
Based on the results in \cite{reiter2} we deduce some topological results concerning holomorphic mappings of Levi-nondegenerate hyperquadrics under biholomorphic equivalence. We study the class $\FTwo$ of so-called nondegenerate and transversal holomorphic mappings sending locally the sphere in $\C^2$  to a Levi-nondegenerate hyperquadric in $\C^3$, which contains the most interesting mappings. We show that from a topological point of view there is a major difference when the target is the sphere or the hyperquadric with signature $(2,1)$. In the first case $\FTwo$ modulo the group of automorphisms is discrete in contrast to the second case where this property fails to hold. Furthermore we study some basic properties such as freeness and properness of the action of automorphisms fixing a given point on $\FTwo$ to obtain a structural result for a particularly interesting subset of $\FTwo$.
\end{abstract}

\maketitle

\section{Introduction and Results}
\label{sec:intro}

We study holomorphic mappings between the sphere $\Sphere{2} \subset \C^2$ and the hyperquadric $\Hyperquadric{\eps}{3} \subset \C^3$, which for $\eps = \pm 1$ is given by $\Hyperquadric{\pm}{3}  \coloneqq \bigl\{(z_1,z_2,z_3) \in \C^3:~|z_1|^2 + |z_2|^2 \pm |z_3|^2 = 1 \bigr\}$, so that $\Hyperquadric{+}{3} = \Sphere{3}$ is the sphere in $\C^3$. Faran \cite{faran} classified holomorphic mappings between spheres in $\C^2$ and $\C^3$ and Lebl \cite{lebl1} classified mappings sending $\Sphere{2}$ to $\Hyperquadric{-}{3}$. In \cite{reiter2} we give a new CR-geometric approach to reprove Faran's and Lebl's results in a unified manner. Let us introduce the following equivalence relation. For $k=1,2$ let $H_k:U_k \rightarrow \C^3$ be a holomorphic mapping where $U_k$ is an open and connected neighborhood of $p_k \in \Sphere{2}$ and $H_k(U_k \cap \Sphere{2}) \subset \Hyperquadric{\eps}{3}$. We say $H_1$ is \textit{equivalent} to $H_2$ if there exist automorphisms $\phi$ and $\phi'$ of $\Sphere{2}$ and $\Hyperquadric{\eps}{3}$ respectively such that $H_2 = \phi' \circ H_1 \circ \phi^{-1}$.

\begin{theorem}[label=theorem:MyTheorem,name={\cite[Theorem 1.3]{reiter2}}]
Let $p \in \Sphere{2}$, $U \subset \C^2$ be an open and connected neighborhood of $p$ and $H: U \rightarrow \C^3$ a non-constant holomorphic mapping satisfying $H(U \cap \Sphere{2}) \subset \Hyperquadric{\eps}{3}$. Then $H$ is equivalent to exactly one of the following maps:
\begin{compactitem}
\item[\rm{(i)}] $\MapMainTheoremTwo{1}{\eps}(z,w) = (z,w,0)$
\item[\rm{(ii)}] $\MapMainTheoremTwo{2}{\eps}(z,w) = \Bigl(z^2,\frac{ (1- \eps + z (1+ \eps)) w}{\sqrt{2}},w^2\Bigr)$
\item[\rm{(iii)}] $\MapMainTheoremTwo{3}{\eps}(z,w) = \Bigl(z, \frac{(1-\eps+z^2(1+\eps)) w}{2 z}, \frac{ (1- \eps +z(1+\eps )) w^2}{2 z} \Bigr)$
\item[\rm{(iv)}] $\MapMainTheoremTwo{4}{\eps}(z,w) = \frac{\left(4 z^3, ( 3(1- \eps) +(1+3 \eps) w^2) w, \sqrt{3} (1 -\eps + 2 (1 +\eps) w + (1-\eps) w^2 ) z\right)}{1 + 3 \eps + 3 (1-\eps) w^2}$
\end{compactitem}
Additionally for $\eps = -1$ we have:
\begin{compactitem}
\item[\rm{(v)}] $H_5(z,w) = \Bigl(\frac{(2 + \sqrt{2} z)z}{1+\sqrt{2} z + w}, w,  \frac{(1+\sqrt{2} z - w)z}{1+\sqrt{2} z+w}\Bigr)$
\item[\rm{(vi)}] $H_6(z,w) =\frac{ \left((1-w) z, 1+ w - w^2,(1+ w) z \right)} {1 - w - w^2}$
\item[\rm{(vii)}] $H_7(z,w) = \bigl(1,h(z,w),h(z,w)\bigr)$ for some non-constant holomorphic function $h: \C^2 \rightarrow \C$
\end{compactitem}
\end{theorem}

In fact we study holomorphic mappings between the Heisenberg hypersurface $\HeisenbergOne{2} \subset \C^2$ and $\HeisenbergTwo{\eps}{3}$, where $\HeisenbergTwo{+}{3}=\HeisenbergOne{3}$ is the Heisenberg hypersurface in $\C^3$. The hypersurfaces $\HeisenbergOne{2}$ and $\HeisenbergTwo{\eps}{3}$ are biholomorphic to $\Sphere{2}$ and $\Hyperquadric{\eps}{3}$ respectively, except one point, and are given by
\begin{align*}
\HeisenbergOne{2} =  \bigl\{(z,w) \in \C^2:~\im w = | z|^2 \bigr\}, \qquad \quad \HeisenbergTwo{\eps}{3} = \bigl\{(z_1',z_2', w') \in \C^3:~\im w' = |z_1'|^2 + \eps |z_2'|^2 \bigr\}.
\end{align*}
We denote by $\FTwo$ the class of germs of $2$-nondegenerate transversal mappings sending a small piece of $\HeisenbergOne{2}$ to $\HeisenbergTwo{\eps}{3}$ and is introduced in more details in \autoref{def:F2} below. This is, in some sense, the most natural and interesting class of mappings when studying holomorphic mappings between $\HeisenbergOne{2}$ to $\HeisenbergTwo{\eps}{3}$. From \cite{reiter2} we know that $\FTwo$ consists of mappings belonging to the orbits of the maps listed in (ii)--(vi) of \autoref{theorem:MyTheorem} with respect to the equivalence relation of automorphisms introduced above, after composing with an appropriate Cayley transform. A direct consequence of the above theorem and some intermediate classification result from \cite{reiter2} is the following topological property of the quotient space of $\FTwo$ modulo automorphisms.

\begin{theorem}[label=theorem:basicTopProps2]
The quotient space $\FTwo/_{\approx}$ with respect to the equivalence relation of \autoref{theorem:MyTheorem} is discrete for $\eps = +1$. This property fails to hold for $\eps = -1$.
\end{theorem}

The above result was not known before and shows one major difference between holomorphic mappings from the sphere in $\C^2$ to the sphere in $\C^3$ and to the hyperquadric with signature $(2,1)$ in $\C^3$. For a germ of a real-analytic CR-submanifold $(M,p)$ of $\C^N$ we write $\Aut_p(M,p)$ for germs of real-analytic CR-diffeomorphisms fixing $p$, which we refer to as \textit{isotropies} of $(M,p)$. Let us denote by $\Isotropies \coloneqq \Auto(\HeisenbergOne{2},0) \times  \Auto(\HeisenbergTwo{\eps}{3},0)$ the direct product of the groups of isotropies of $(\HeisenbergOne{2},0)$ and $(\HeisenbergTwo{\eps}{3},0)$ respectively, which we introduce in \autoref{def:isotropies} below in more details. We study the action of $\Isotropies$ on $\FTwo$  given by $\Isotropies \times \FTwo \rightarrow \FTwo, (\phi,\phi',H) \mapsto \phi' \circ H \circ \phi^{-1}$. We write $\FreeFTwo \subset \FTwo$ for the set of maps which have only trivial stabilizers. The action is called \textit{proper} if the associated map $(\phi,\phi',H) \mapsto (H,\phi' \circ H \circ \phi^{-1})$ is a proper map. The following result holds:

\begin{theorem}[label=theorem:propertiesAction]
The mapping $N: \Isotropies \times \FreeFTwo \rightarrow \FreeFTwo$ given by $N(\phi,\phi',H) \coloneqq \phi' \circ H \circ \phi^{-1}$ is a free and proper left action.
\end{theorem}

Based on this result we obtain the following theorem concerning the real-analytic structure of $\FreeFTwo$, where $\Pi: \FreeFTwo \rightarrow \FreeNTwo$ denotes the normalization map induced by the mapping $N$ and $\FreeNTwo$ denotes a particular set of representatives of the quotient $\FreeFTwo / \Isotropies$ to be given below in \autoref{lemma:StabilizerN2}:

\begin{theorem}[label=theorem:realAnalyticStructureTop]
If $\eps = +1$ then $\Pi: \FreeFTwo \rightarrow \FreeFTwo/\Isotropies$ is a real-analytic principal fibre bundle with structure group $\Isotropies$. If $\eps = -1$ then $\FreeFTwo$ is locally mapped to $\Isotropies \times \FreeNTwo$ via local real-analytic diffeomorphisms. In particular $\FreeFTwo$ is not a smooth manifold.
\end{theorem}

Note that the second part of \autoref{theorem:realAnalyticStructureTop} stands in contrast to the case of $\Aut_p(M,p)$. Assuming some nondegeneracy conditions for certain germs of real-analytic CR-submanifolds $(M,p)$, such as Levi-nondegeneracy, it is known that $\Aut_p(M,p)$ admits a manifold structure, see \cite{BER97}, \cite{BER99}, \cite{BRWZ04}, \cite{kowalski}, \cite{KZ05}, \cite{LM07}, \cite{LMZ08} and \cite{JL2013}. To prove \autoref{theorem:realAnalyticStructureTop} we use a real-analytic version of the so-called local slice theorem for free and  proper actions. For proper smooth actions of non-compact Lie groups the first proof of the local slice theorem was given in \cite[2.2.2 Proposition]{palais}. In the real-analytic setting a global slice theorem was proved by \cite[section VI]{HHK} and \cite[Theorem 0.6]{IK}. We also obtain the following result about the different topologies we can associate to the quotient space $\FreeFTwo/\Isotropies$:

\begin{theorem}[label=theorem:quotientTop]
The quotient topology $\tau_Q$ on $\FreeFTwo/\Isotropies \simeq \FreeNTwo$ coincides with the induced topology $\tau_J$ of $\FreeFTwo$, which carries the topology induced by the jet space $J^3_0(\HeisenbergOne{2},\HeisenbergTwo{\eps}{3})$.
\end{theorem}

We organize this paper as follows: We introduce the necessary notations, tools and results in \Autoref{sec:Preliminaries}. In \Autoref{sec:HomeomorphicNormalForm} we study different normal forms with respect to isotropies and in \Autoref{section:BasicTopProps} we investigate the connectedness of $\FTwo$ and discreteness of the quotient space. In the remaining sections we study properties of the action of the group of isotropies on $\FTwo$, which finally result in some structural and topological information of $\FreeFTwo$ and $\FreeFTwo/\Isotropies$ respectively in \Autoref{sec:Structure}. This article is partly based on the author's thesis \cite{reiter1} at the University of Vienna. Some computations are carried out with \textit{Mathematica 7.0.1.0} \cite{wolfram}.

\section{Preliminaries}
\label{sec:Preliminaries}

\begin{definition}
We fix coordinates $(z,w)=(z_1,\ldots, z_n,w) \in \C^{n+1}$. For $h: \C^{n+1} \rightarrow \C$ a holomorphic function $h(z,w) = \sum_{\alpha, \beta} a_{\alpha \beta} z^{\alpha} w^{\beta}$ defined near $0$ we write $\bar{h}(\bar z,\bar w) \coloneqq \overline{h(z,w)}=\sum_{\alpha, \beta} \bar{a}_{\alpha \beta} \bar z^{\alpha} \bar w^{\beta}$  for the complex conjugate of $h$. Derivatives of $h$ with respect to $z$ or $w$ we denote by $h_{z^{\alpha} w^{\beta}}(0)\coloneqq  \frac{\partial^{|\alpha|+ |\beta|} h}{z^{\alpha} w^{\beta}}(0)$. For $n\geq 1$ and a map $H: \C^{n+1} \rightarrow \C^{n'+1}$ defined near $0$ with components $H= \bigl(f_1,\ldots,f_{n'},g \bigr)$ we write $H_{z^{\alpha} w^{\beta}}(0) = \bigl (f_{1 z^{\alpha} w^{\beta}}(0),\ldots, f_{n' z^{\alpha} w^{\beta}}(0),g_{z^{\alpha} w^{\beta}}(0) \bigr)$.
\end{definition}

\subsection{Classes of Maps, Automorphisms and Equivalence Relations}

\begin{definition}
We write $\mathcal H(p;p') \coloneqq \{H:(\C^N,p) \rightarrow (\C^{N'},p'): ~ H \text{ holomorphic} \}$ for the \textit{set of germs of holomorphic mappings from $(\C^N,p)$ to $(\C^{N'},p')$}. For $(M,p) \subset \C^N$ and $(M',p') \subset \C^{N'}$ germs of real-analytic hypersurfaces we denote by 
\begin{align*}
\mathcal H(M,p;M',p') \coloneqq \{H \in \mathcal H(p;p'): H(M \cap U) \subset M' \text{ for some neighborhood $U$ of $p$}\},
\end{align*}
the \textit{set of germs of holomorphic mappings from $(M,p)$ to $(M',p')$}.
\end{definition}

\begin{definition}[label=def:isotropies]
\begin{compactitem} 
\item[\rm{(i)}] The collection of germs of locally real-analytic CR-diffeomorphisms of $(M,p)$ we denote by $\Aut(M,p) \coloneqq \{H: (\C^N,p) \rightarrow \C^N: H \text{ holomorphic}, H(M) \subset M, \det(H'(p)) \neq 0 \}$ and the \textit{group of isotropies of $(M,p)$ fixing $p$} by $\Aut_p(M,p) \coloneqq \{H \in \mathcal \Aut(M,p): H(p)=p\}$.
\item[\rm{(ii)}] We write $\R^+\coloneqq\{x \in \R:~x>0\}$, denote the unit sphere in $\C$ by $\UnitSphere \coloneqq \{e^{\imu t}: 0 \leq t < 2 \pi \}$ and set $\Gamma \coloneqq \R^+ \times \R \times \UnitSphere \times \C$. For an element $\sigma_{\gamma} \in \Auto(\HeisenbergOne{2},0)$ we denote $\gamma=(\lambda, r,u,c) \in \Gamma$ and write:
\begin{align}
\label{eq:sigma}
\sigma_{\gamma}(z,w) \coloneqq \frac{(\lambda u (z + c w), \lambda^2 w)}{1 - 2 \imu \bar c z + (r - \imu |c|^2) w}.
\end{align}
\item[\rm{(iii)}]
We define for $\sigma = \pm 1$ if $\eps = -1$ and $\sigma = +1$ if $\eps=+1$
\begin{align}
\label{eq:setS}
\mathcal S^2_{\eps,\sigma} \coloneqq \big\{a' =(a_1', a_2') \in \C^2:  |a_1'|^2 + \eps |a_2'|^2 = \sigma \big\},
\end{align}
and let
\begin{align}
\label{eq:UTwoOneOne}
U' \coloneqq \left(\begin{array}{cc}
u' a_1' & - \eps u' a_2'  \\
\bar a_2' & \bar a_1' \end{array}\right),  \qquad u' \in \UnitSphere, \quad a'=(a_1',a_2') \in \mathcal S^2_{\eps,\sigma}.
\end{align}
We set $\Gamma' \coloneqq \R^+ \times \R \times \UnitSphere \times \mathcal S^2_{\eps,\sigma} \times \C^2$ to denote elements $\sigma'_{\gamma'} \in\Auto(\HeisenbergTwo{\eps}{3},0)$ via $\gamma' =(\lambda',r',u',a',c') \in \Gamma'$, where $c'=(c_1',c_2')$:
\begin{align}
\label{eq:tau}
\sigma'_{\gamma'}(z',w') \coloneqq \frac{(\lambda' U'~{^t(z' +c' w')}, \sigma {\lambda'}^2 w')}{1 - 2 \imu (\bar c_1' z_1' + \eps \bar c_2' z_2') + \bigl(r' - \imu (|c_1'|^2 + \eps |c_2'|^2) \bigr) w'}.
\end{align}
\item[\rm{(iv)}] We call elements of $\Gamma \times \Gamma'$ \textit{standard parameters}. If the standard parameters $(\gamma,\gamma') \in \Gamma\times \Gamma'$ are chosen such that $(\sigma_{\gamma},\sigma'_{\gamma'}) = (\id_{\C^2}, \id_{\C^3})$, we say the standard parameters are \textit{trivial}.
\end{compactitem}
\end{definition}

 \begin{definition}[label=definition:localEquivalence]
For $G,H \in \mathcal H(M,p;M',p')$ we denote the following equivalence relation: $G \sim H :\Leftrightarrow  \exists~ (\phi,\phi') \in \Aut_p(M,p) \times  \Aut_{p'}(M',p')$: $G = \phi' \circ H \circ \phi^{-1}$. The equivalence classes in $\mathcal H(M,p;M',p')/_{\sim}$ are denoted by $[F] \coloneqq \{G \in \mathcal H(M,p;M',p'): G \sim F\}$.\\
In the case where $(p,p')=(0,0)$ and $(M,M')=(\HeisenbergOne{2},\HeisenbergTwo{\eps}{3})$ we call the above relation \textit{isotropic equivalence} and write $O_0(H)$ for the orbit of a map $H$, called the \textit{isotropic orbit of H}.
\end{definition}

\subsection{The Class \texorpdfstring{$\FTwo$}{F2}, the Normal Form \texorpdfstring{$\NTwo$}{N2} and its Classification}

In \cite{reiter2} we introduced the following class of mappings, which are $2$-nondegenerate and transversal. These mappings represent the immersive maps, which are not equivalent to the linear embedding, see \cite[Proposition 2.16]{reiter2}.

\begin{definition}[label=def:F2]
For a neighborhood $U \subset \C^2$ of $0$ let us denote the set $\FTwo(U)$ of holomorphic mappings $H=(f_1,f_2,g)$ with $H(U \cap \HeisenbergOne{2}) \subset \HeisenbergTwo{\eps}{3}$, which satisfy $H(0) =0$, $f_{1 z} (0) f_{2 z^2} (0) -f_{2 z} (0)f_{1 z^2} (0) \neq 0$ and $g_w(0) >0$. We denote by $\FTwo$ the set of germs $H$, such that $H \in \FTwo(U)$ for some neighborhood $U \subset \C^2$ of $0$.
\end{definition}

\begin{proposition}[name={\cite[Propostion 3.1]{reiter2}},label=proposition:NormalForm2Nondeg]
Let $H \in \FTwo$. Then there exist isotropies $(\sigma ,\sigma') \in \Isotropies$ such that $\widehat H \coloneqq \sigma' \circ H \circ \sigma^{-1}$ satisfies $\widehat H (0) =0$ and the following conditions:
\begin{multicols}{3}
\begin{compactitem}
\item[\rm{(i)}] $\widehat H_z(0) =(1,0,0)$
\item[\rm{(ii)}] $\widehat H_w(0) = (0,0,1)$
\item[\rm{(iii)}] $\widehat f_{2 z^2}(0) = 2$
\item[\rm{(iv)}] $\widehat f_{2 z w}(0) = 0$
\item[\rm{(v)}] $\widehat f_{1 w^2}(0) = |\widehat f_{1 w^2}(0)| \geq 0$
\item[\rm{(vi)}] $\re\bigl(\widehat g_{w^2}(0)\bigr) = 0$
\item[\rm{(vii)}] $\re\bigl(\widehat f_{2 z^2 w}(0) \bigr) = 0$
\end{compactitem}
\end{multicols}
A holomorphic mapping of $\FTwo$ satisfying the above conditions is called a normalized mapping. The set of normalized mappings is denoted by $\NTwo$.
\end{proposition}

\begin{remark}[label=remark:SummaryJetNormalizedMapping]
A mapping $H \in \NTwo$ necessarily satisfies the following conditions, see \cite[Remark 3.4]{reiter2}:
\begin{multicols}{2}
\begin{compactitem}
\item[\rm{(i)}]  $H(0)= (0,0,0)$
\item[\rm{(ii)}]  $H_z(0) = (1,0,0)$
\item[\rm{(iii)}]  $H_w(0) = (0,0,1)$
\item[\rm{(iv)}] $H_{z^2}(0) = (0,2,0)$
\item[\rm{(v)}] $H_{zw}(0) = (\frac{\imu \eps} 2 ,0,0)$
\item[\rm{(vi)}]  $H_{w^2}(0) = (|f_{1 w^2}(0)|, f_{2 w^2}(0),0)$
\item[\rm{(viii)}] $H_{z^2 w}(0) = (4 \imu |f_{1 w^2}(0)|,\imu \im(f_{2 z^2 w}(0)), 0)$
\end{compactitem}
\end{multicols}
\end{remark} 

We classify all mappings belonging to $\NTwo$ in \cite{reiter2}.

\begin{theorem}[name={\cite[Theorem 4.1]{reiter2}},label=theorem:ReductionOneParameterFamilies2]
The set $\NTwo$ consists of the following mappings, where $s \geq0$:
\begin{align*}
\MapOneParameterFamilyTwo{1}{\eps}(z,w)\coloneqq&~ \Bigl(2 z (2 + \imu \eps w),4 z^2,4 w\Bigr)/(4-w^2),\\
\MapOneParameterFamilyThree{2}{s}{\eps}(z,w)\coloneqq& ~\Bigl(4 z - 4 \eps s z^2 + \imu (\eps  - s^2) z w  + s w^2,  4 z^2 + s^2 w^2,w (4- 4 \eps s z - \imu  (\eps +s^2) w) \Bigr) \\
& \quad  \slash \Bigl( 4 - 4 \eps s z - \imu (\eps+ s^2) w   - 2 \imu s z w-  \eps s^2 w^2 \Bigr),\\
\MapOneParameterFamilyThree{3}{s}{\eps}(z,w)\coloneqq & ~\Bigl(256 \eps z  + 96 \imu z w + 64 \eps s w^2 + 64 z^3 + 64 \imu \eps s z^2 w - 3 (3 \eps - 16 s^2) z w^2  + 4 \imu s w^3,\\
& \quad  256 \eps z^2 - 16 w^2 + 256s z^3 + 16 \imu z^2 w - 16 \eps s z w^2 - \imu \eps w^3, \\
& \quad w(256 \eps - 32 \imu w + 64 z^2  - 64 \imu \eps s z w  -  (\eps + 16 s^2) w^2)\Bigr) \\
& \quad \slash  \Bigl( 256 \eps - 32 \imu w + 64 z^2 - 192 \imu \eps s z w  - (17 \eps + 144 s^2) w^2  + 32 \imu \eps z^2 w + 24 s z w^2\\
& \quad + \imu w^3 \Bigr).
\end{align*}
Each mapping in $\NTwo$ is not isotropically equivalent to any different mapping in $\NTwo$.
\end{theorem}

For $\eps = -1$ we can give the following picture of $\NTwo$ in the $s$-parameter space according to \autoref{theorem:ReductionOneParameterFamilies2}, see \cite[\S 4]{reiter2} for more details: 
\begin{figure}[H]
\begin{center}
\begin{tikzpicture}
		\node (0) at (-1.75, -1) {};
		\node at (-2.1,-1) {$\MapOneParameterFamilyThree{2}{0}{-}$};
		\node  (1) at (1.4, 1.25) {};
		\node at (0,-0.7) {$\MapOneParameterFamilyThree{2}{\frac 1 2}{-}$};
		\node  (2) at (1.75, 1.75) {};
		\node  (3) at (-1.50, 1.25) {};
		\node at (1.1,0.35) {$\MapOneParameterFamilyThree{2}{1}{-}$};
		\node  (4) at (1.75, -1) {};
		\node at (2.2,-1) {$\MapOneParameterFamilyThree{3}{0}{-}$};
		\node  (5) at (-1.75, 1.75) {};
		\draw [bend right=15,line width= 1pt] (0.center) to (1.center);
		\draw [bend right=1,looseness=1,dotted,line width= 1pt] (1.center) to (2.center);
		\draw [bend left, looseness=0.75,line width= 1pt] (4.center) to (3.center);
		\draw [bend left=3,dotted,line width= 1pt](3.center) to (5.center);
		\fill (-2,0) circle (1.25pt) node[above] {$\MapOneParameterFamilyTwo{1}{-}$};		
		\node (6) at (-9, -1) {};
		\node at (-9.4,-1) {$\MapOneParameterFamilyThree{2}{0}{+}$};
		\node  (7) at (-7, 1) {};
		\node  (8) at (-6.5, 1.5) {};
		\node  (9) at (-8, -1) {};
		\node at (-8.3,-1.1) {$\MapOneParameterFamilyThree{3}{0}{+}$};
		\node  (10) at (-6, 1) {};
		\node (11) at (-5.5,1.5){};
		\draw [line width= 1pt] (6.center) to (7.center);
		\draw [dotted,line width= 1pt] (7.center) to (8.center);
		\draw [line width= 1pt] (9.center) to (10.center);
		\draw [dotted,line width= 1pt](10.center) to (11.center);
		\fill (-9,0) circle (1.25pt) node[above] {$\MapOneParameterFamilyTwo{1}{+}$};
\end{tikzpicture}
\caption{$\NTwo$ for $\eps = \pm1$ in the parameter space} 
\label{figure:orbits}
\end{center}
\end{figure} 

An immediate consequence of the normalization and classification of maps in $\FTwo$ is the following jet determination result.
 
\begin{corollary}[label=cor:jetDetermination,name={\cite[Corollary 4.8]{reiter2}}]
Let $U \subset \C^2$ be a neighborhood of $0$ and $H: U \rightarrow \C^3$ a holomorphic mapping. We denote the components of $H$ by $H=(f,g)=(f_1,f_2,g)$ and write $j_0(H) \coloneqq \{j_0^2(H), f_{z^2 w}(0)\}$. If for $H_1,H_2 \in \FTwo$ the coefficients belonging to $j_0(H_1)$ and $j_0(H_2)$ coincide, we have $H_1 \equiv H_2$.
\end{corollary}

\subsection{Associated Topologies}

We deal with the following topologies, see e.g. \cite{BER97}.

\begin{definition}[label=def:inductiveLimitTop]
For $K \subset \C^N$ a compact neighborhood of $p \in \C^N$ we denote the Frech\'et space $\mathcal H_K(p;p')$ of germs of holomorphic mappings, defined in a neighborhood of $K$, which map $p \in \C^N$ to $p' \in \C^{N'}$. The topology for $\mathcal H_K(p;p')$ is given by uniform convergence on compact sets. We equip $\mathcal H(p;p')$ with the inductive limit topology, denoted by $\tau_C$, with respect to Frech\'et spaces $\mathcal H_K(p;p')$, where $K$ is some compact neighborhood of $p$ in $\C^N$. Then for $H,H_n \in \mathcal H(p;p')$ we say that $H_n$ converges to $H$, if there exists $K \subset \C^N$ a compact neighborhood of $p$, such that each $H_n$ is holomorphic in a neighborhood of $K$ and $H_n$ converges uniformly to $H$ on $K$. \\
For $\mathcal H(M,p;M',p') \subset \mathcal H(p;p')$ we consider the induced topology of $\mathcal H(p;p')$ denoted by $\tau_C$.
\end{definition}

\begin{definition}
Let $H:\C^N \rightarrow \C^{N'}$ be a holomorphic mapping defined at $p \in \C^N$ and $\alpha \in \N^N$. We denote by $j^k_pH$ the \textit{$k$-jet of $H$ at $p$} defined as
\begin{align*}
j_p^k H \coloneqq \left( \frac{\partial^{|\alpha|} H} {\partial Z^{\alpha} } (p): |\alpha| \leq k\right). 
\end{align*}
We denote by $J^k_{p,p'}$ the collection of all $k$-jets at $p$ of germs of mappings from $(\C^N,p)$ to $(\C^{N'},p')$. We set $J^k_p \coloneqq J^k_{p,p}$. Let $(M,p) \subset (\C^N,p)$ and $(M',p') \subset (\C^{N'},p')$ be germs of submanifolds. For $k\in\N$ we denote by $J^k_q(M,p;M',p')$ the \textit{space of $k$-jets of $\mathcal H(M,p;M',p')$ at $q$}. We write $J_q^k(M,p) \coloneqq J^k_q(M,p;M,p)$ and $J_0^k(M;M') \coloneqq J^k_0(M,0;M',0)$. We denote by $G^k_p(M,p) \subset J^k_p(M,p)$ the \textit{space of $k$-jets of $\Aut_p(M,p)$ at $p$}.
\end{definition}

\begin{remark}
Note that $J^k_p(M,p;M',p') \subset J^k_{p,p'}$. We identify $J^k_{p,p'}$ with the space of germs of holomorphic polynomial mappings from $\C^N$ to $\C^{N'}$ up to degree $k$, which map $p \in \C^N$ to $p' \in \C^{N'}$. Thus $J^k_{p,p'}$ can be identified with some $\C^K$, where $K \coloneqq N' \binom{N+ k}{N}$, such that the topology for $J^k_{p,p'}$, denoted by $\tau_J$, is induced by the natural topology of $\C^K$. We refer to the topology $\tau_J$ as \textit{topology of the jet space}.
\end{remark}

\begin{definition}
We say $\mathcal K \subset \mathcal H(M,p;M',p')$ admits a \textit{jet parametrization for $\mathcal K$ of order $k$} if the following properties hold: There exists a mapping $\Psi: \C^N \times \C^K \supset U \rightarrow \C^{N'}$, with $K = N' \binom{N+ k}{N}$ from above and $U$ an open neighborhood of $\{p\} \times J_p^k(M,p;M',p')$, which is holomorphic in the first $N$ variables, real-analytic in the remaining $K$ variables, such that $F(Z) = \Psi(Z, j_p^k F)$, for all $F \in \mathcal K$.
\end{definition}

\begin{remark}
If $\mathcal K \subset \mathcal H(M,p;M',p')$ admits a jet parametrization of some order $k$, then $\tau_C = \tau_J$, which follows from the real-analyticity in the last $K$ variables.
\end{remark}

\begin{remark}[label=rem:TCequalsTJ]
 Based on \cite[Proposition 25, Corollary 26--27]{lamel} we obtain a jet parametrization of order $4$ for $\mathcal K = \FTwo$ in \cite[Lemma 4.3]{reiter2} and by \autoref{cor:jetDetermination} we have that $K = K_0 \coloneqq 15$. Using \autoref{theorem:ReductionOneParameterFamilies2} and the notation from \autoref{cor:jetDetermination} we identify $\FTwo$ with a subset $\JetFTwo \subset \C^{K_0}$ given by $\JetFTwo \coloneqq \{j_0(H): H\in \FTwo\}$ and the topology we use in the sequel for $\FTwo$ is $\tau_J$. 
\end{remark}

\begin{definition}[label=def:quotientTop]
Let $X$ be topological spaces, $Y$ a set and $q: X \rightarrow Y$ a surjective mapping. We call the topology on $Y$ induced by $q$ the \textit{quotient topology $\tau_Q$ on $Y$}, where a set $U \subset Y$ is open in $Y$ if $q^{-1}(U)$ is open in $X$. 
\end{definition}

\section{Homeomorphic Variations of Normal Forms}
\label{sec:HomeomorphicNormalForm}

\begin{definition}[label=def:NormalForm]
Let $\mathcal H$ be a subset of $\mathcal H(M,p;M',p')$. A proper subset $\mathcal K \subsetneq \mathcal H$ is called \textit{normal form for $\mathcal H$}, if for each $[F] \in \mathcal H /_{\sim}$, there exists a unique representative $G \in \mathcal K \cap [F]$. We denote the mapping which assigns to each $H \in \mathcal H$ the representative $G \in \mathcal K \cap [H]$ as $\pi: \mathcal H \rightarrow \mathcal K$. A normal form $\mathcal K$ for $\mathcal H$ is called \textit{admissible} if $\pi: \mathcal H \rightarrow \mathcal K$ is continuous.
\end{definition}

\begin{remark}
The uniqueness of the representative $F\in \mathcal K \cap [F]$ in \autoref{def:NormalForm} is not a restriction: Assume we have another representative $F\neq G \in \mathcal K$ in the class $[F]$, then $G$ is equivalent to $F$, hence it suffices to choose only one element from the set of all representatives which belong to $\mathcal K \cap [F]$. If there exists an admissible normal form $\mathcal K$ for $\mathcal H$ we observe that in each orbit of any not necessarily admissible normal form $\mathcal K'$ for $\mathcal H$, there exists an element of $\mathcal K \cap \mathcal K'$.
\end{remark}

\begin{theorem}[label=thm:differentNormalForm]
Let $\mathcal N'$ be an admissible normal form for $\FTwo$. Then $\mathcal N'$ is homeomorphic to $\NTwo$, where we equip $\mathcal N'$ and $\NTwo$ with $\tau_J$.
\end{theorem}

\begin{proof}
Let us denote by $\pi': \FTwo \rightarrow \mathcal N'$ the continuous mapping as in \autoref{def:NormalForm}.  We note that the class $\NTwo$ introduced in \autoref{proposition:NormalForm2Nondeg} is an admissible normal form for $\FTwo$ as in \autoref{def:F2}. To see this we need to inspect the proof of \autoref{proposition:NormalForm2Nondeg} given in \cite[Proposition 3.1]{reiter2}. If we equip $\FTwo$ with $\tau_J$ we obtain that for the isotropies $(\sigma,\sigma') \in \Isotropies$ deduced in the proof, the mapping $\pi: \FTwo \rightarrow \NTwo, H \mapsto \sigma' \circ H \circ \sigma^{-1}$ iis continuous, since the isotropies depend real-analytically on $j_0^3(H)$. Hence we have the following diagram:
\begin{figure}[H]
\begin{tikzpicture}[node distance=1.5cm]
  \node (A) {$\FTwo$};
  \node (B) [below of = A, left of=A] {$\mathcal N'$};
  \node (C) [below of = A, right of=A] {$\NTwo$};
  \draw[->] (A.210) to node[above=0.2cm] {$\pi'$} (B.60);
   \draw[->] (B.30) to node[below=0.2cm] {$\incl'$} (A.240);
  \draw[->] (A.330) to node [above=0.2cm] {$\pi$} (C.120);
   \draw[->] (C.150) to node[below=0.2cm] {$\incl$} (A.300);
 \draw[->] (B) to node [below=0.1cm] {$\psi$} (C);
\end{tikzpicture}
\caption{Diagram for admissible normal forms}
\end{figure}

The mapping $\incl':\mathcal N' \rightarrow \FTwo$ is the inclusion mapping, which is given by $\incl'(H) \coloneqq H$ for all $H\in \mathcal N'$ and similar for $\incl: \NTwo \rightarrow \FTwo$. The map $\psi:\mathcal N' \rightarrow \NTwo$ is given by $\psi(H) \coloneqq F$ for $H\in \mathcal N'$ and $F \in \NTwo \cap [H]$.
Since $\mathcal N'$ and $\NTwo$ are normal forms, we obtain that $\psi$ is a bijective mapping. Further since $\psi = \pi \circ \incl'$ and $\psi^{-1} = \pi' \circ \incl$ are compositions of continuous mappings, we obtain that $\psi$ is a homeomorphism.
\end{proof}

\begin{example}[label=ex:differentNormalForm]
Starting with $\NTwo$ we can construct different admissible normal forms $\mathcal N'$ as follows: We fix a pair of isotropies $(\phi_0,\phi_0') \in \Isotropies$ and consider the isotropies $(\hat \phi,\hat \phi') \in \Isotropies$ from \autoref{proposition:NormalForm2Nondeg}, such that $\pi:\FTwo \rightarrow \NTwo$ is given by $\pi(H) \coloneqq \hat \phi' \circ H \circ {\hat \phi}^{-1}$, denoted by $\widehat H$. We define $\phi \coloneqq \phi_0 \circ \hat\phi$ and $\phi' \coloneqq \phi'_0 \circ {\hat \phi}'$, to obtain for $F\in \FTwo$,
\begin{align*}
 \phi' \circ F \circ \phi^{-1} = \phi'_0 \circ {\hat \phi'} \circ F \circ {\hat \phi}^{-1} \circ \phi_0^{-1} = \phi'_0 \circ \widehat F \circ \phi_0^{-1},
\end{align*}
where $\widehat F \in \NTwo$. We define $\mathcal N' \coloneqq \bigl\{ \phi'_0 \circ \widehat F \circ \phi_0^{-1}:  \widehat F \in \NTwo \bigr\}$. Since $\hat \phi$ and ${ \hat \phi}'$ depend continuously on $F \in \FTwo$, the mapping $\pi': \FTwo \rightarrow \mathcal N'$ given by $\pi'(F) \coloneqq \phi' \circ F \circ \phi^{-1}$ is continuous, which means that $\mathcal N'$ is an admissible normal form. 
\end{example}

\section{A Topological Property of the Quotient Space of \texorpdfstring{$\FTwo$}{F2}}
\label{section:BasicTopProps}

\begin{lemma}[label=lemma:basicTopProps1]
The class $\FTwo$ consists of $\frac{5 + \eps} 2$ connected components.
\end{lemma}

\begin{proof}
We denote by $\pi: \FTwo \rightarrow \NTwo$ the continuous map, which takes $F \in \FTwo$ to $\widehat F \in \NTwo$ according to \autoref{proposition:NormalForm2Nondeg}. For $k=2,3$ we set $C_k \coloneqq \{G_{k,s}^{\eps}: s\geq 0\}$ and $\widehat \NTwo \coloneqq C_2 \cup C_3$. The space of standard parameters $\Gamma \times \Gamma'$ is path-connected, since for mappings in $\FTwo$ and $\eps = -1$ we only consider isotropies as in \eqref{eq:tau} with $\sigma = +1$. Thus for any $H \in \NTwo$ the isotropic orbit $O_0(H)$ is path-connected. First we treat the case $\eps = -1$. We observe that $\widehat \FTwo = \bigcup_{H \in \widehat N} O_0(H)$ is path-connected. If $\FTwo$ would be connected then $\pi(\FTwo) = \NTwo$ is connected, which is not the possible, since $\NTwo$ consists of $2$ connected components $\MapOneParameterFamilyTwo{1}{-}$ and $\widehat \NTwo$. Thus $\FTwo$ has $2$ connected components $O_0(\MapOneParameterFamilyTwo{1}{-})$ and $\widehat \FTwo$. For $\eps = +1$ we note that the set $O_0(C_k) \coloneqq \bigcup_{H \in C_k} O_0(H)$ for $k=2,3$ is path-connected and $\NTwo$ consists of $3$ connected components. Thus $\FTwo$ admits at most $3$ connected components. $\FTwo$ is not connected since then $\pi(\FTwo) = \NTwo$ would be connected. If $\FTwo$ consists of $2$ connected components $\FTwo_1, \FTwo_2$, such that $\FTwo = \FTwo_1 \cup \FTwo_2$, we need to distinguish several cases. Either $\FTwo_1 = O_0(\MapOneParameterFamilyTwo{1}{+}) \cup O_0(C_k)$, and $\FTwo_2=O_0(C_{\ell})$, where $k \neq \ell$ and $k,\ell \in \{2,3\}$ or $\FTwo_1 = O_0(C_2) \cup O_0(C_3)$ and $\FTwo_2 = O_0(\MapOneParameterFamilyTwo{1}{+})$. In all cases we have by the continuity of $\pi$, that $\pi(\FTwo_1)$ is connected, which is not possible.
\end{proof}

\begin{proof}[Proof of \autoref{theorem:basicTopProps2}]
The set $X \coloneqq \FTwo/_{\approx}$ consists of elements denoted by $[F]$ for $F\in \FTwo$. We equip $X$ with the quotient topology such that the canonical projection $\pi: \FTwo \rightarrow X$ is continuous. For $\eps = +1$ we have $X =\{\MapOneParameterFamilyTwo{1}{+}, \MapOneParameterFamilyThree{2}{0}{+}, \MapOneParameterFamilyThree{3}{0}{+} \}$ by our classification. By \autoref{lemma:basicTopProps1} we obtain that $\pi^{-1}(H)$ for $H \in X$ is a connected component of $\FTwo$, hence open. Thus $X$ carries the discrete topology. To prove the statement for $\eps = -1$ we write $H_0 \coloneqq \MapOneParameterFamilyThree{2}{1/2}{-} \in \NTwo$ and $H_1 \coloneqq \MapOneParameterFamilyThree{3}{0}{-} \in \NTwo$. For $k=0,1$ let $U_k \in X$ be an open neighborhood of $[H_k]$, then $V_k \coloneqq \pi^{-1}(U_k)$ is an open neighborhood of the orbit of $H_k$ in $\FTwo$. According to our classificiation there exists a sequence $(G_n)_{n\in \N}$ of mappings in $\FTwo$, where each $G_n \in [H_1]$ and $G_n \rightarrow H_0$ in $\FTwo$ as $n \rightarrow \infty$. Thus there exists $N\in \N$ such that $G_n \in V_0 \cap V_1$ for all $n \geq N$, which shows $[H_1] \in U_0 \cap U_1$ and completes the proof.
\end{proof}

\section{Isotropic Stabilizer}

\begin{lemma}[label=lemma:StabilizerN2]
We set $\FreeNTwo \coloneqq \NTwo \setminus \{\MapOneParameterFamilyTwo{1}{\eps}, \MapOneParameterFamilyThree{2}{0}{\eps}, \MapOneParameterFamilyThree{3}{0}{\eps}\}$ and $\FreeFTwo \coloneqq \bigcup_{H \in \FreeNTwo} O_0(H)$. The isotropic stabilizer $\stab_0(H)\coloneqq \{(\phi,\phi') \in \Isotropies: \phi' \circ H \circ \phi^{-1} = H\}$ of $H$ is trivial for $H\in \FreeNTwo$. Furthermore we have $\stab_0(\MapOneParameterFamilyTwo{1}{\eps}) = \stab_0(\MapOneParameterFamilyThree{2}{0}{\eps})$ is homeomorphic to $\Sphere{1}$ and $\stab_0(\MapOneParameterFamilyThree{3}{0}{\eps})$ is homeomorphic to $\Z_2$.
\end{lemma}

\begin{proof}
We let $H =(f,g)= (f_1,f_2,g) \in \NTwo$ satisfy the conditions we collected in \autoref{remark:SummaryJetNormalizedMapping}. We write $s\coloneqq 2 |f_{1 w^2}(0)| \geq 0, x \coloneqq f_{2 w^2}(0) \in \C$ and $y\coloneqq \im\bigl(f_{2 z^2 w}(0)\bigr) \in \R$. By \autoref{cor:jetDetermination} we only need to consider coefficients in $j_0(H)$. We let $(\sigma,\sigma') \in \Isotropies$ with the notation from \eqref{eq:sigma}, \eqref{eq:UTwoOneOne} and \eqref{eq:tau} respectively and consider the equation
\begin{align}
\label{eq:freenessEquation}
\sigma' \circ H \circ \sigma^{-1} = H,
\end{align}
where we parametrize $\sigma^{-1}$ as in \eqref{eq:sigma}. The coefficients of order $1$, which are $f_z(0)$ and $H_w(0)$, are given by $U'~ {^t(u \lambda \lambda',0)}  = (1,0)$ and $U' ~{^t(u c + \lambda c_1', \lambda c_2',\lambda \lambda')}  = (0,0,1)$. These equations imply $\lambda' = 1/\lambda, a_2' = c_2' = 0, a_1' = 1/(u u')$ and $c_1' = - u c/\lambda$. Assuming these standard parameters we consider the coefficients of order $2$, which are $f_{z^2}(0), H_{zw}(0)$ and $H_{w^2}(0)$, given by:
\begin{align}
\label{eq:freenessSecondOrderZ2}
(0,2 u' u^3 \lambda) & = (0,2),\\
\label{eq:freenessSecondOrderZW}
\left(-r - \lambda^2 r' + \imu \eps \lambda^2/ 2, 2 u' u^3 \lambda c,0\right) & =\left(\imu \eps/2,0,0\right),\\
\label{eq:freenessSecondOrderW2}
\left(\lambda^2 (\lambda s + \imu \eps u c)/u,u u' \lambda (\lambda^2 x + 2 u^2 c^2), -2 (r + \lambda^2 r')\right) & = (s,x,0).
\end{align}
The second component of \eqref{eq:freenessSecondOrderZW} implies $c=0$. If we assume this value for $c$ we obtain for the third order terms $f_{z^2 w}(0)$ the following equation:
\begin{align}
\label{eq:freenessThirdOrderZ2W}
\left(2 \imu u \lambda^3 s, u' u^3  \lambda (-4 r - 2 \lambda^2 r' + \imu \lambda^2 y) \right) & = (4 \imu s, \imu y).
\end{align}
The second component of \eqref{eq:freenessSecondOrderZ2} shows $\lambda=1$. Furthermore we obtain from the third component of \eqref{eq:freenessSecondOrderW2} that $r'= - r$ and since from the second component of \eqref{eq:freenessSecondOrderZ2} we get $u' u^3 = 1$, which uniquely determines $u'$, we obtain from the second component of \eqref{eq:freenessThirdOrderZ2W} that $r=0$. The remaining equation from the first component of \eqref{eq:freenessSecondOrderW2}, which comes from the coefficient $f_{1 w^2}(0)$, is $s/u = s$. If $s > 0$ we obtain that $u=1$ and hence all standard parameters are trivial, which proves the first claim of the lemma. If $s=0$, then $H \in\{\MapOneParameterFamilyTwo{1}{\eps}, \MapOneParameterFamilyThree{2}{0}{\eps}, \MapOneParameterFamilyThree{3}{0}{\eps}\}$, since these maps are precisely the one satisfying $f_{1 w^2}(0) = 0$ in the list of mappings from \autoref{theorem:ReductionOneParameterFamilies2}. It is easy to check that the isotropic stabilizers of the maps $\MapOneParameterFamilyTwo{1}{\eps}$ and $\MapOneParameterFamilyThree{2}{0}{\eps}$ are generated by the isotropies $(\sigma(z,w),\sigma'(z_1',z_2',w')) = (u z,w, z_1'/u,z_2'/u^2,w')$ with $|u|= 1$. If we consider $ \MapOneParameterFamilyThree{3}{0}{\eps}$ in \eqref{eq:freenessEquation}, then we obtain that $(\sigma(z,w),\sigma'(z_1',z_2',w')) = (\delta z,w,\delta z_1',z_2',w')$, where $\delta = \pm 1$, are the only elements of $\stab_0(\MapOneParameterFamilyThree{3}{0}{\eps})$, which proves the last claim of the lemma.
\end{proof}

\section{Properties of the Group Action}

\begin{lemma}[label=lemma:CharacterizationProper,name={\cite[Proposition 3.20]{tomdieck}}]
Let $G$ be a topological group acting freely on a topological space $X$ via the action $\alpha: G \times X \rightarrow X$. 
Then the following statements are equivalent:
\begin{compactitem}
\item[\rm{(i)}] $G$ acts properly.
\item[\rm{(ii)}] Let $\alpha': G \times X \rightarrow X \times X$ be given by $\alpha'(g,x) \coloneqq \bigl(x, \alpha(g,x) \bigr)$. The image $C \subset X \times X$ of $\alpha'$ is closed and the map $\varphi_{\alpha}: C \rightarrow G$, given by $\varphi_{\alpha}\bigl(x, \alpha(g,x) \bigr) \coloneqq g$ is continuous.
\end{compactitem}
\end{lemma}

\begin{remark}[label=rem:NotationGeneralNormalMap]
For $F:(\C^2,0) \rightarrow (\C^3,0)$ a germ of a holomorphic mapping, for which we assume that $F\in \FTwo$ and the jet $j_0(F) \subset j_0^3 F$ is of the form as in \autoref{remark:SummaryJetNormalizedMapping}, we write $F=(f^1,f^2,f^3)$ for the components and denote derivatives of $F$ at $0$ by $f^k_{\ell m} \coloneqq f^k_{z^{\ell} w^m}(0)$.
\end{remark}

\begin{lemma}[label=lemma:SequenceN2Action]
For $n \in \N$ we let $H_n,H \in \FreeNTwo$ and $(\phi_n,\phi'_n) \in \Isotropies$ such that $\phi_n' \circ H_n \circ \phi_n^{-1} \rightarrow H$, then $(\phi_n,\phi'_n) \rightarrow (\id_{\C^2},\id_{\C^3})$ and $H_n \rightarrow H$ as $n \rightarrow \infty$.
\end{lemma}

\begin{proof}
We assume for $H_n=(h^1_n,h^2_n,h^3_n)$ and $H=(h^1,h^2,h^3)$ to be given as in \autoref{rem:NotationGeneralNormalMap}, where in $H_n$ the coefficients depend on $n\in \N$. We write $s_n\coloneqq 2 |h^1_{n02}| \in \R^+, x_n \coloneqq h^2_{n 0 2} \in \C$ and $y_n\coloneqq \im\bigl(h^2_{n21} \bigr)$. To each $(\phi_n,\phi'_n) \in \Isotropies$ we associate $(\gamma_n ,\gamma_n') \in \Gamma \times \Gamma'$ respectively, where we use the notation for the parametrization of $\Isotropies$ from \eqref{eq:sigma} and \eqref{eq:tau}. According to \autoref{theorem:ReductionOneParameterFamilies2} we have that $H_n$ depends on $s_n > 0$. Let us denote $\Xi \coloneqq \Gamma \times \Gamma' \times \R^+$ and write $\xi_n=(\gamma_n,\gamma_n',s_n) \in \Xi$. We define $\Psi_{n} \coloneqq \phi_n' \circ H_n \circ \phi_n^{-1}$, which depends on $\xi_n \in \Xi$. For components of $\Psi_n$, we write $\Psi_n=(\psi^1_n,\psi^2_n,\psi^3_n)$ and $\psi_n=(\psi^1_n,\psi^2_n)$. Limits are always considered when $n \rightarrow \infty$. We start with the first order terms of $\Psi_n$. We let $U_n'$ be the $2\times 2$-matrix from \eqref{eq:UTwoOneOne} with entries $u_n', a_{1n}'$ and $a_{2n}'$ instead of $u', a_1'$ and $a_2'$, then we have
\begin{align}
\label{eq:sequenceFirstOrderZ}
\psi_{n z}(0) = ~&  \lambda_n \lambda'_n U_n' ~{^t(u_n,0)},\\
\label{eq:sequenceFirstOrderW}
\Psi_{n w}(0) = ~ & \lambda_n \lambda_n' \Bigr(U_n'~{^t(u_n c_n + \lambda_n c_{1 n}',\lambda_n c_{2 n}'),\lambda_n \lambda'_n}\Bigl).
\end{align}
Since $\psi^3_{n w}(0) \rightarrow 1$ we obtain $\lambda_n \lambda_n' \rightarrow 1$, which implies if we consider \eqref{eq:sequenceFirstOrderZ}, since $\psi_{n z}(0) \rightarrow (1,0)$, that $u_n u_n' a_{1 n}'  \rightarrow 1$ and $a_{2 n}' \rightarrow 0$. Because $a_n'=(a_{1n}',a_{2n}') \in \mathcal S_{\eps,\sigma}^2$ from \eqref{eq:setS}, we have $|a_{1n}'| \rightarrow 1$. If we consider the first two components in \eqref{eq:sequenceFirstOrderW} we obtain from $\psi_{n w}(0) \rightarrow (0,0)$ and $(|a_{1n}'|,|a_{2 n}'|) \rightarrow (1,0)$ that $u_n c_n + \lambda_n c_{1 n}' \rightarrow 0$ and $c_{2n}' \rightarrow 0$. Next we consider the second order terms of $\Psi_n$ to obtain
\begin{align}
\label{eq:sequenceSecondOrderZ2}
\psi_{n z^2}(0) = 
2 u_n \lambda_n \lambda_n' U_n'
{^t( 2 \imu \bigl(\bar c_n + u_n \lambda_n \bar c_{1 n}' \bigr), u_n \lambda_n)},
\end{align}
where the left-hand side of \eqref{eq:sequenceSecondOrderZ2}, $\psi_{n z^2}(0)$, must converge to $(0,2)$. After applying ${U_n'}^{-1}$ we rewrite the second components of \eqref{eq:sequenceSecondOrderZ2} as 
\begin{align}
\label{eq:sequenceSecondOrderZ2second}
2 u_n^2  \lambda_n^2 \lambda_n'  = a_{1 n}'  \Bigl(-\bar a_{2 n}' \psi^1_{n z^2}(0)/(u_n' a_{1 n}' ) +  \psi^2_{n z^2}(0) \Bigr),
\end{align}
where since $(|a_{1n}'|,|a_{2 n}'|) \rightarrow (1,0)$ the absolute value of the right-hand side of \eqref{eq:sequenceSecondOrderZ2second} converges to $2$. Taking the absolute value of the left-hand side of \eqref{eq:sequenceSecondOrderZ2second} implies $\lambda_n \rightarrow 1$, which together with $\lambda_n \lambda_n' \rightarrow 1$ shows $\lambda_n' \rightarrow $1. Next we consider
\begin{align}
\label{eq:sequenceSecondOrderZW}
\psi_{n z w}(0) = 
\frac {\imu} 2 \lambda_n \lambda_n' U_n' {^t ( T_1(\gamma_n,\gamma_n'),4 \lambda_n (c_{2 n}' \bigl(\bar c_n  +u_n \lambda_n \bar c_{1 n}' \bigr) -  \imu  u_n^2 c_n ))},
\end{align}
where the real-analytic function $T_1: \Gamma \times \Gamma' \rightarrow \C$ does not depend on $a_n' \in \mathcal S_{\eps,\sigma}^2$ and $u_n'$. The left-hand side of \eqref{eq:sequenceSecondOrderZW} has to converge to $\bigl(\imu \eps / 2,0\bigr)$ and we rewrite the second component of \eqref{eq:sequenceSecondOrderZW} as
\begin{align}
\label{eq:sequenceSecondOrderZWsecond}
4  \lambda_n \Bigl(c_{2 n}' \bigl(\bar c_n  +u_n \lambda_n \bar c_{1 n}' \bigr) -  \imu  u_n^2 c_n \Bigr)  = -2 \imu  \Bigl(-\bar a_{2 n}' \psi^1_{n z w}(0) +  u_n' a_{1 n}' \psi^2_{n z w}(0) \Bigr)/(\lambda_n \lambda_n' u_n').
\end{align}
Taking the limit, we know since $(|a_{1n}'|,|a_{2 n}'|) \rightarrow (1,0)$ and $(\lambda_n,\lambda_n') \rightarrow (1,1)$, that the right-hand side of \eqref{eq:sequenceSecondOrderZWsecond} converges to $0$ and if we also use $u_n c_n + \lambda_n c_{1 n}' \rightarrow 0$ we obtain that 
$c_n \rightarrow 0$, such that $c_{1n}' \rightarrow 0$. Next we compute
\begin{align}
\label{eq:sequenceSecondOrderW2third}
\psi^3_{n w^2}(0) = 2 \lambda^2_n {\lambda'_n}^2 \Bigl(- (r_n + \lambda_n^2 r_n') + \imu \Bigl(c_n \bar c_n + \eps \lambda_n^2 c_{2n}' \bar c_{2 n'} + \lambda_n \bar c_{1 n}' \bigl(2 u_n c_n + \lambda_n c_{1 n}' \bigr)\Bigr)\Bigr).
\end{align}
We let $n \rightarrow \infty$ and take all the previously obtained limits of the sequences $c'_n=(c_{1n}',c_{2 n}') \in \C^2, c_n$ and $\lambda_n, \lambda_n'$,
then we have since $\psi^3_{n w^2}(0)\rightarrow 0$, that $r_n +  \lambda_n^2 r_n'\rightarrow 0$. Next we compute 
\begin{align}
\label{eq:sequenceSecondOrderW2}
\psi_{n w^2}(0) = \lambda_n \lambda_n' U_n'  {^t( \lambda_n^3 s_n + T_2(\gamma_n,\gamma_n'), \lambda_n^3 x_n + T_3(\gamma_n,\gamma_n'))},
\end{align}
where $T_2, T_3: \Gamma \times \Gamma' \rightarrow \C$ are real-analytic functions and $T_2$ is given by
\begin{align*}
T_2(\gamma_n,\gamma_n') = ~&  2(u_n c_n + c_{1 n}' \lambda_n) \bigl( \imu |c_n|^2 - r_n - \lambda_n^2 r_n' \bigr) + 2 \imu \lambda_n \bar c_{1 n}' (u_n c_n + \lambda_n c_{1 n}') (2 u_n c_n + \lambda_n c_{1 n}') \\
& + \imu \eps \lambda_n^2 \bigl(u_n c_n(1+ 2 |c_{2 n}'|^2) + 2 \lambda_n c_{1 n}'  |c_{2 n}'| \bigr),
\end{align*}
such that $T_2(\gamma_n,\gamma_n') \rightarrow 0$. Then the first component of \eqref{eq:sequenceSecondOrderW2} becomes
\begin{align}
\lambda_n^3 s_n + T_2(\gamma_n,\gamma_n') =  \Bigl(\bar a_{1n}'  \psi^1_{n w^2}(0) + \eps u_n' a_{2 n}' \psi^2_{n w^2}(0) \Bigr)/(\lambda_n \lambda_n' u_n').
\end{align}
Since $\bigl(\psi^1_{n w^2}(0),\psi^2_{n w^2}(0)\bigr) \rightarrow (2 |h^1_{02}|,h^2_{02}) \in \R^+ \times \C$ we obtain that $\bar a_{1n}'/u_n' \rightarrow 1$ and $s_n \rightarrow 2 |h^1_{02}|$. Then $u_n u_n' a_{1 n}'  \rightarrow 1$ implies that $u_n \rightarrow 1$ and further inspection of \eqref{eq:sequenceSecondOrderZ2second} gives $u_n^2 / a_{1 n}' \rightarrow 1$, which shows $a_{1n}' \rightarrow 1$ and $u_n' \rightarrow 1$. Finally we consider 
\begin{align}
\label{eq:sequenceThirdOrder}
\psi_{n z^2 w}(0) = \lambda_n \lambda_n' U_n' 
\left( 
\begin{array}{c}
- 4 \imu u_n^2 \lambda_n^3 s_n + T_4(\gamma_n,\gamma_n') \\ 
-2 \eps u_n^2 \lambda_n (2 r_n + \lambda_n^2 r_n') + \imu \eps u_n^2 \lambda_n^3 y_n + 6 u_n^3 \lambda_n^2 c_n s_n + T_5(\gamma_n,\gamma_n')
\end{array}
\right),
\end{align}
where $T_4, T_5: \Gamma \times \Gamma' \rightarrow \C$ are real-analytic functions and $T_5$ is given by
\begin{align*}
T_5(\gamma_n,\gamma_n') = ~& 2 \imu \eps \lambda_n \Bigl(4 \imu \bar c_n c_{2n}' (\bar c_n + 2 u_n \lambda_n \bar c_{1n}' )  + 2 c_n u_n^2 (5 \bar c_n + 3  u_n \lambda_n \bar c_{1n}') \\
& +  u_n^2 \lambda_n^2 \bigl(|c_{1n}'|^2  + 3 \eps |c_{2n}'|^2 + 4 \imu \bar c_{1n}' c_{2 n}'\bigr)\Bigr),
\end{align*}
hence $T_5(\gamma_n,\gamma_n') \rightarrow 0$. Since $(\psi^1_{nz^2 w}(0),\psi^2_{nz^2 w}(0)) \rightarrow (2 \imu |h^1_{02}|,\imu h^2_{21}) \in \imu \R \times \imu \R$ we obtain if we consider the real part of the second component of \eqref{eq:sequenceThirdOrder}, that  $2 r_n + r_n' \rightarrow 0$, which together with $r_n +  \lambda_n^2 r_n'\rightarrow 0$ shows $(r_n,r_n' ) \rightarrow (0,0)$. To sum up we obtain $(\phi_n,\phi_n') \rightarrow (\id_{\C^2},\id_{\C^3})$, which completes the proof.
\end{proof}

\begin{proof}[Proof of \autoref{theorem:propertiesAction}]
First we observe that $N$ is a continuous map from $\Isotropies\times \FreeFTwo$ to $\FreeFTwo$, since the image of $N$ consists of rational mappings, which depend real-analytically on the jets of the isotropies and the mapping. By construction $N$ is a left action and \autoref{lemma:StabilizerN2} shows that $N$ restricted to $\FreeNTwo$ is a free action. Next we assume the general case $H\in \FreeFTwo$ and consider the equation $\phi' \circ H \circ \phi^{-1} = H$ for $(\phi,\phi') \in \Isotropies$. We can write $H = \widehat {\phi'} \circ \widehat H \circ {\widehat {\phi}}^{-1}$, where $\widehat H\in \FreeNTwo$ and $(\widehat {\phi},\widehat {\phi}') \in \Isotropies$ are unique according to \autoref{lemma:StabilizerN2}. After setting $(\psi,\psi')=({\widehat {\phi}}^{-1} \circ \phi \circ \widehat {\phi},{\widehat {\phi'}}^{-1}  \circ \phi' \circ \widehat \phi' )$ we rewrite $\phi' \circ H \circ \phi^{-1} = H$ as $\psi' \circ \widehat H \circ \psi^{-1} = \widehat H$. Since $N$ acts freely on $\FreeNTwo$ we obtain that $(\psi,\psi') = (\id_{\C^2}, \id_{\C^3})$ and the freeness of the action. To show the properness of $N$ we prove (ii) of \autoref{lemma:CharacterizationProper}. We let the mapping $N': \Isotropies \times \FreeFTwo \rightarrow \FreeFTwo \times \FreeFTwo$ be given by $N'(\phi,\phi',H) \coloneqq \bigl(H,N(\phi,\phi',H)\bigr)$. Then we know from \autoref{proposition:NormalForm2Nondeg} that the image $C$ of $N'$ agrees with $\FreeFTwo \times \FreeFTwo$, which is closed in $\FreeFTwo \times \FreeFTwo$. Next we let the mapping $\varphi_N: C \rightarrow \Isotropies$ be given by $\varphi_N \bigl(H, N(\phi,\phi',H)\bigr) \coloneqq (\phi,\phi')$. To show the continuity of $\varphi_N$ we let $(H_n)_{n \in \N}\in \FreeFTwo$ be a sequence of mappings with
\begin{align}
\label{eq:isoHnHatH}
H_n  \rightarrow H\in \FreeFTwo, \qquad \text{and} \qquad \phi_n' \circ H_n \circ \phi_n^{-1}  \rightarrow \check H \in \FreeFTwo.
\end{align}
Using \autoref{proposition:NormalForm2Nondeg} we assume w.l.o.g. $H\in \FreeNTwo$. Moreover by \autoref{proposition:NormalForm2Nondeg} we write $\check H = \phi' \circ \widehat H \circ \phi^{-1}$ for $\widehat H \in \FreeNTwo$. Then we need to conclude that $(\phi_n, \phi_n') \rightarrow (\phi,\phi')$ and $H = \widehat H$, which implies the continuity of $\varphi_N$. For each $n \in \N$ we write $H_n = \widehat{\phi}_n' \circ \widehat H_n \circ {\widehat \phi_n}^{-1}$, where $\widehat H_n\in \FreeNTwo$. If we substitute the above representations of $H_n$ and $\check H$ into \eqref{eq:isoHnHatH} we obtain $\varphi_n' \circ \widehat H_n \circ \varphi_n^{-1} \rightarrow \widehat H \in \FreeNTwo$, where $(\varphi_n,\varphi_n')=(\phi^{-1} \circ \phi_n \circ {\widehat \phi_n} ,{\phi'}^{-1} \circ \phi_n' \circ \widehat{\phi}_n')$. By \autoref{lemma:SequenceN2Action} we have $(\varphi_n,\varphi_n') \rightarrow (\id_{\C^2}, \id_{\C^3})$. Again applying \autoref{lemma:SequenceN2Action} to $H_n \rightarrow H\in \FreeNTwo$ shows $(\widehat{\phi}_n,\widehat \phi_n') \rightarrow (\id_{\C^2},\id_{\C^3})$ to obtain $(\phi_n,\phi_n') \rightarrow (\phi,\phi')$ as required.
\end{proof}

\section{On the Real-Analytic Structure of \texorpdfstring{$\FreeFTwo$}{FreeF2}}
\label{sec:Structure}

\begin{lemma}[label=lemma:preimageMF]
Let $\Pi: \FreeFTwo \rightarrow \FreeNTwo$ be given by $\Pi(H) \coloneqq \phi' \circ H \circ \phi^{-1}$, where $(\phi,\phi') \in \Isotropies$ are the unique isotropies according to \autoref{proposition:NormalForm2Nondeg} and \autoref{lemma:StabilizerN2}. For $k=2,3$ we write $M_{k,\eps} \coloneqq \{\Pi^{-1}(\MapOneParameterFamilyThree{k}{s}{\eps}): s > 0\}$. Then $M_{k,\eps}$ is a real-analytic real submanifold of $\FreeFTwo$ of real dimension $16$.
\end{lemma}

\begin{proof}[Proof of \autoref{lemma:preimageMF}]
For fixed $k=2,3, s > 0$ and $\delta >0$ we write $G_{\delta,s} \coloneqq \{\MapOneParameterFamilyThree{k}{t}{\eps}: t \in B_{\delta}(s) \cap \R^+\}$, where $B_{\delta}(s) \coloneqq \{t \in \R^+:|t-s|< \delta \}$. To prove the lemma we show that for every $s_0\in\R^+$ and sufficiently small $\delta_0>0$ there exists a locally real-analytic parametrization for $M \coloneqq \Pi^{-1}(G_{\delta_0,s_0})$. As noted in \autoref{rem:TCequalsTJ} we identify $\FTwo$ with the set $\JetFTwo \subset \C^{K_0}$. \autoref{theorem:ReductionOneParameterFamilies2} implies that for each $H\in M$ there exist $(\phi,\phi') \in \Isotropies, k\in\{2,3\}$ and $s_1 \in B_{\delta_0}(s_0) \cap \R^+$, such that $H= \phi' \circ \MapOneParameterFamilyThree{k}{s_1}{\eps} \circ \phi^{-1}$. This fact is used to describe $M$ locally via parametrizations as follows: For $s> 0$ sufficiently near $s_0$ let $F_s$ be a mapping as in \autoref{rem:NotationGeneralNormalMap}, which depends real-analytically on $s\coloneqq 2|f^1_{02}|$. For the remaining coefficients in $j_0(F_s)$ we write $x \coloneqq f^2_{02}$ and $y\coloneqq \im\bigl(f^2_{21} \bigr)$, where we suppress the dependence on $s$ notationally. We use the real version of the notation for the parametrization of $\Isotropies$ as in \eqref{eq:sigma} and \eqref{eq:tau}. Here we denote the set of real parameters of $\Auto(\HeisenbergOne{2},0)$ by $\Gamma$ and of $\Auto(\HeisenbergTwo{\eps}{3},0)$ by $\Gamma'$. Let us denote $\Xi \coloneqq \Gamma \times \Gamma' \times \R^+ \subset \R^{N_0}$, where $N_0 \coloneqq 16$. For $\xi=(\gamma,\gamma',s)\in \Xi$ we define the mapping
\begin{align}
\label{eq:parametrizationPreimage}
\Psi: \Xi \rightarrow \JetFTwo,  \qquad \Psi(\xi) \coloneqq j_0(\phi'_{\gamma'} \circ F_s \circ \phi_{\gamma}^{-1}),
\end{align}
where we use the notation as in \eqref{eq:sigma} and \eqref{eq:tau} for $\phi_{\gamma}$ and $\phi'_{\gamma'}$ respectively and suppress the dependence on  $\eps$. We set $\check \Psi(z,w) \coloneqq \bigl(\phi'_{\gamma'} \circ F_s \circ \phi_{\gamma}^{-1}\bigr) (z,w)$ with components $\check \Psi= (\check\psi^1,\check\psi^2,\check\psi^3)$ and $\check \psi \coloneqq (\check\psi^1,\check\psi^2)$. The holomorphic mapping $\check \Psi$ is defined in a small neighborhood $U \subset \C^2$ of $0$ and satisfies $\check \Psi(\HeisenbergOne{2} \cap U) \subset \HeisenbergTwo{\eps}{3}$. By \autoref{theorem:ReductionOneParameterFamilies2} and the real-analytic dependence of the isotropies on the standard parameters, we note that $\Psi$ and $\check \Psi$ are real-analytic in $\xi\in \Xi$. We assume w.l.o.g. that $\xi_0$ is chosen in such a way that $(\phi_{\gamma},\phi'_{\gamma'})=(\id_{\C^2},\id_{\C^3})$. Consequently we write $O(2)$ for terms involving standard parameters of the isotropies which vanish to second order at $\xi_0$. Moreover since we only consider $a_1' \in \C$ near $1$ and $a'=(a_1',a_2')\in \mathcal S_{\eps,\sigma}^2$ from \eqref{eq:setS}, we substitute $\bar a_1' = (1- \eps |a_2'|^2)/a_1'$ into $\Psi$, which is then given by the following expressions:
\begin{align*}
\check \psi_{z}(0) = & \bigl(u u' \lambda \lambda' a_1', u \lambda \lambda' \bar a_2'\bigr),\\
\check \Psi_w(0)  = & \Bigl(u' \lambda \lambda' a_1' (u c + \lambda c_1'), \lambda^2 \lambda' c_2'/a_1', \lambda^2 {\lambda'}^2 \Bigr) + O(2), \\
\check \psi_{z^2}(0)  = & \Bigl(2 \imu u u' \lambda \lambda' (\imu \eps u \lambda a_2' + 2 (\bar c + u \lambda \bar c_1') a_1' ),  2 u^2 \lambda^2 \lambda' / a_1' \Bigr) + O(2),\\
\check \Psi_{z w}(0) = & \Bigl(- \frac 1 2 u u' \lambda \lambda' a_1' (2 (r+ \lambda^2 r') - \imu \eps \lambda^2), u \lambda^2 \lambda' \left(\frac {\imu \eps} 2 \lambda \bar a_2'+ \frac{2 u c}{a_1'} \right), 2 \imu \lambda^2 {\lambda'}^2 (\bar c + u \lambda \bar c_1') \Bigr) \\
& + O(2),\\
\check \Psi_{w^2}(0) = & \Bigl(u' \lambda^3 \lambda' \bigl(a_1' (\imu \eps u c + \lambda s) - \eps \lambda a_2'  x \bigr), \lambda^4 \lambda' \bigl(x/a_1' + \bar a_2' s \bigr), - 2 \lambda^2 {\lambda'}^2 (r + \lambda^2 r')\Bigr) + O(2),\\
\check \psi_{z^2 w}(0) = & \Bigl(-u u' \lambda^3 \lambda' \Bigl(4 a_1'\bigl(- \imu u \lambda s + \eps(\bar c + u \lambda \bar c_1')\bigr) + \imu \eps u \lambda a_2' y \Bigr), \\
& \quad u^2 \lambda^2 \lambda' \Bigl( \Bigl(-2 (2 r + \lambda^2 r') +6 \eps u \lambda c s + \imu \lambda^2 y \Bigr)/a_1' + 2 \imu \lambda^2 \bar a_2' s \Bigr) \Bigr) + O(2).
\end{align*}
In a first step we show that for given $\xi_0 \in \Xi$ the Jacobian of $\Psi$ with respect to $\xi$ evaluated at $\xi_0$, denoted by $\Psi_{\xi}(\xi_0)$, is of full rank $N_0$. But instead of considering the real equations of $\Psi$, we conjugate $\Psi$ and compute the Jacobian of the system $\Phi \coloneqq(\Psi, \overline{\Psi})\in \C^{2 K_0}$, with respect to $\xi=(u,\lambda,c,r, u', a_1',a_2',\lambda',c_1',c_2',r',s; \bar c,\bar a_2',\bar c_1',\bar c_2')\in \C^{N_0}$ and evaluate at
\begin{align}
\label{eq:xizero}
\xi_0=(1,1,0,0,1,1,0,1,0,0,0,s_0;0,0,0,0)\in \R^{N_0},
\end{align}
denoted by $\Phi_{\xi}(\xi_0)$. We bring the transpose of $\Phi_{\xi}(\xi_0)$ into echelon form, where we denote the resulting matrix by $\varphi=(\varphi^1,\ldots,\varphi^{N_0})$, where $\varphi^j=(\varphi^j_1,\ldots, \varphi^j_{2 K_0}) \in \C^{2 K_0}, 1 \leq j \leq N_0$, such that $\rank\bigl(\Phi_{\xi}(\xi_0)\bigr) = \rank(\varphi)$. In the following we suppress the evaluation of $\Phi$ at $\xi_0$ notationally and perform elementary row operations. The matrix given by
\begin{align*}
(\varphi^1,\ldots, \varphi^{11}) \coloneqq  & \quad \Bigl(\Phi_u,\Phi_{\bar a_2'},\Phi_{c_1'}, \Phi_{c_2'} , \Phi_{\lambda}, \Phi_{\bar c}, \Phi_{a_1'}, \Phi_{r'}, \Phi_{c},\Phi_{a_2'}, \Phi_{s} \Bigr) \\
& -\Bigl(0,0,0,\Phi_u,\Phi_u,0,\Phi_u,0,\Phi_{c_1'}, \imu \eps/2 \Phi_{\bar c},0\Bigr),
\end{align*}
is in row echelon form, with constant nonzero entries in the main diagonal. Each $0$ above represents $0 \in \C^{2 K_0}$. Next we define
\begin{align*}
\varphi^{12} \coloneqq  \Phi_{\lambda'} + \Phi_u/3 - \Phi_{\lambda} - \Phi_{a_1'}/3 - \imu \eps/8 \Phi_{r'} + 10 s_0/3 \Phi_s, \varphi^{13} \coloneqq  \Phi_{u'} - \Phi_u/3 - 2/3 \Phi_{a_1'} - 2/3 \Phi_{s},
\end{align*}
which are of the following form, where we denote by $h'$ derivatives of a function $h$ depending on $s$ with respect to $s$:
\begin{align*}
\varphi^{12} = & ~  (0,\ldots,0,\varphi^{12}_{12}, \ldots, \varphi^{12}_{2 K_0}) \\
= & \left(0,\ldots,0, \frac{-2(4 x - 5 s_0 x')} 3, 2 \imu \eps, \frac{8\imu s_0}  3, \frac{2 \imu (3 \eps - 3 y + 5 s_0 y')}3,-\frac 1 3, \varphi^{12}_{17}, \ldots \varphi^{12}_{2 K_0}\right)\\ 
\varphi^{13} = & ~ (0,\ldots,0,\varphi^{13}_{12}, \ldots, \varphi^{13}_{2 K_0}) = \left(0,\ldots,0, \frac{2 x - s_0 x'} 3,0,-\frac{8 \imu s_0} 3,-\frac{\imu s_0 y'} 3,-\frac 2 3,\varphi^{13}_{17}, \ldots \varphi^{13}_{2 K_0}\right).
\end{align*}
Then we define $\varphi^{14} \coloneqq  \Phi_r - \Phi_{r'}, \varphi^{15} \coloneqq  \Phi_{\bar c'_2}, \varphi^{16} \coloneqq \Phi_{\bar c'_1}$, and compute $\varphi^{14} = -2 (e_{15} +e_{2 K_0}), \varphi^{15} = e_{19}, \varphi^{16} = -2 e_{24}+ \imu \eps e_{26} -12 \eps s e_{2 K_0}$, where for $j \in \N$ we denote by $e_j$ the $j$-th unit vector in $\R^{2 K_0}$. We have to consider several cases. If $\varphi^{12}_{12} \neq 0$, we consider $\tilde \varphi^{13} \coloneqq \varphi^{13}-\varphi^{13}_{12} \varphi^{12}/\varphi^{12}_{12}$, such that $\tilde \varphi^{13}_{13}$ is a multiple of $-2 x + s_0 x'$. If $\tilde \varphi^{13}_{13} \neq 0$, then $\varphi=(\varphi^1,\ldots,\varphi^{12},\tilde \varphi^{13}, \varphi^{14},\varphi^{15},\varphi^{16})$ is in echelon form. If $\tilde \varphi^{13}_{13} =0$, then $x = C s^2$, where $C\in \C \setminus\{ 0\}$ and we have $\tilde \varphi^{13}_{14} \neq 0$, which again implies that $\varphi=(\varphi^1,\ldots,\varphi^{12},\tilde \varphi^{13}, \varphi^{14},\varphi^{15},\varphi^{16})$ is in echelon form. Next we treat $\varphi^{12}_{12} = 0$. First we consider the trivial case. If $x=0$, then since $s_0>0$, we have $x'=0$ and $\varphi=(\varphi^1,\ldots,\varphi^{16})$ is in echelon form. Now we assume $x\neq 0$ which implies $x' \neq 0$ and we solve $\varphi^{12}_{12} = 0$. The solution is given by $x = C s^{4/5}$, where $C \in \C \setminus \{ 0\}$ and $\varphi=(\varphi^1,\ldots,\varphi^{11}, \varphi^{13}, \varphi^{12}, \varphi^{14},\varphi^{15},\varphi^{16})$ is in echelon form. We sum up that in all cases the Jacobian $\Phi_{\xi}(\xi_0)$ of the system $\Phi$ evaluated at $\xi_0$ is of full rank $N_0$, hence we conclude that $\Psi$ from \eqref{eq:parametrizationPreimage} is a real-analytic locally regular mapping if we choose $\delta_0>0$ sufficiently small in $M$. For $\Psi$ to be a local parametrization of $M$ it remains to show that for each sufficiently small neighborhood $U \subset \Xi \subset \R^{N_0}$ of $\xi_0$, there exists a neighborhood $W \subset \C^{K_0}$ of $\Psi(\xi_0) = F_{s_0}$, such that $\Psi(U) = W \cap M$. We have $\Psi(U) = \{j_0(H): \exists \xi=(\gamma,\gamma',t) \in U: H = {\phi'}^{-1}_{\gamma'} \circ F_t \circ \phi_{\gamma} \}$ and with the notation from the very beginning of this proof for $\delta>0$ we have
\begin{align*}
M = \Pi^{-1}(F_{\delta,s_0}) = & \{H \in \FreeFTwo: \exists(\gamma,  \gamma',s) \in \Gamma \times \Gamma' \times B_{\delta}(s_0) \cap \R^+: {\phi'}_{\gamma'} \circ H \circ \phi_{\gamma}^{-1} = F_s\}.
\end{align*}
\autoref{rem:TCequalsTJ},  together with the fact that for each $H \in M$ we can write $H= {\phi'}_{\gamma'}^{-1} \circ F_s \circ \phi_{\gamma}
$, shows $\Psi(U) \subset M$. We assume that there exists $U \subset \Xi$ a neighborhood of $\xi_0$, such that for any neighborhood $W$ of $\Psi(\xi_0) = F_{s_0}$ we have $\Psi(U) \neq W \cap M$. We choose open, connected neighborhoods $(W_n)_{n\in \N}$ of $F_{s_0}$ with $\bigcap_n W_n = \{F_{s_0}\}$ and $\Psi(U) \neq W_n \cap M$ for all $n \in \N$. There exists a sequence of mappings $(H_n)_{n\in \N} \in \FreeFTwo$ such that $H_n \in W_n \cap M$ and $H_n \not\in \Psi(U)$. We write $H_n = {\phi'}^{-1}_{\gamma_n'} \circ F_{s_n} \circ \phi_{\gamma_n}$ and conclude by \autoref{lemma:SequenceN2Action} that $(\gamma_n,\gamma_n',s_n) \rightarrow \xi_0$ in $\Xi$. Thus eventually $H_n \in \Psi(U)$ for large enough $n\in \N$, which completes the proof of the lemma.
\end{proof}

\begin{remark}[label=rem:realAnalyticLieGroup]
 By \cite[Corollary 1.2]{BER97} the group $\Isotropies$ is a totally real, closed, real-analytic submanifold of  $G^2_0(\HeisenbergOne{2},0) \times G^2_0(\HeisenbergTwo{\eps}{3},0) \subset J_0^2(\HeisenbergOne{2},0) \times J_0^2(\HeisenbergTwo{\eps}{3},0)$. Hence $\Isotropies$ is a real-analytic real Lie group. \cite[Theorem 4]{BM} states that the action of  a real-analytic Lie group $G$ on a real-analytic manifold $M$ is real-analytic, i.e., the map $G \times M \rightarrow M, (g,m) \rightarrow g \cdot m$ is a real-analytic map between real-analytic manifolds. Hence we obtain for $M$ being a real-analytic submanifold, that $N: \Isotropies \times M \rightarrow M$ is a real-analytic action.
 \end{remark}

 \begin{theorem}[name={\cite[Theorem 1.11.4]{dk}},label=thm:localTrivial]
Let $M$ be a real-analytic manifold equipped with an action $G \times M \rightarrow M$, where $G$ is a real-analytic Lie group. Assume that the action is free and proper. Then $M/G$ has the unique structure of a real-analytic manifold of real dimension $\dim_{\R} M - \dim_{\R} G$ and the topology of $M/G$ is the quotient topology $\tau_Q$. We denote by $\varphi: M \rightarrow M / G$ the canonical projection given by $\varphi(m) = G \cdot m \coloneqq \{g \cdot m: g \in G\}$ for $m\in M$. For every $s \in M/G$ there is an open neighborhood $S \subset M/G$ of $s$ and a real-analytic diffeomorphism $\psi: \varphi^{-1}(S) \rightarrow G \times S, \psi: m \mapsto (\psi_1(m),\psi_2(m))$, such that for $m \in \varphi^{-1}(S), g\in G$ we have $\varphi(m) = \psi_2(m)$ and $\psi(g \cdot m) = (g \cdot \psi_1(m),\psi_2(m))$.
\end{theorem}
 
\begin{remark}[label=rem:PFB]
 The above \autoref{thm:localTrivial} says that the triple $(\varphi,M,M/G)$ is a \textit{real-analytic principle fibre bundle with structure group $G$}.
\end{remark}

\begin{proof}[Proof of \autoref{theorem:realAnalyticStructureTop}]
We note that by \autoref{lemma:basicTopProps1} and \autoref{lemma:preimageMF} the set $\FreeFTwo$ is a real-analytic manifold and by \autoref{rem:realAnalyticLieGroup} we know that $\Isotropies$ is a real-analytic Lie group. Thus from \autoref{thm:localTrivial} and \autoref{rem:PFB} the conclusion for $\eps = +1$ follows. Next we show the claim for $\eps = -1$: For k=1,2 we set $N_k \coloneqq  \{\MapOneParameterFamilyThree{k+1}{s}{-}: s > 0\}$ and $N_0 \coloneqq N_1 \cap N_2 = \{\MapOneParameterFamilyThree{2}{1/2}{-}\}$. The corresponding preimages are denoted by $M_k \coloneqq \Pi^{-1}(N_k) \subset \FreeFTwo$, such that $M_0  \coloneqq M_1 \cap M_2 = \Pi^{-1}(N_0)$. We set $M \coloneqq M_1 \cup M_2$. By \autoref{lemma:preimageMF} for $k=1,2$ we have that $M_k$ is a real-analytic submanifold of $\FreeFTwo$. Thus by \autoref{thm:localTrivial} locally $M_k$ is real-analytically diffeomorphic to $\Isotropies \times S_k$, where $S_k$ is a real submanifold with $\dim_{\R}(S_k) = \dim_{\R}(M_k) - \dim_{\R}(\Isotropies) = 1$, by \autoref{lemma:preimageMF}. By \autoref{proposition:NormalForm2Nondeg} it is possible to normalize any element in $S_k$ with unique isotropies which depend real-analytically on elements of $S_k$. Thus, since $\dim_{\R}(N_k) = 1$, we map $S_k$ to $N_k$ via real-analytic diffeomorphisms. We obtain that for $k=1,2$ there exists an open neighborhood $U_k \subset \FreeFTwo$ of $N_0$ and a real-analytic diffeomorphism $\phi_k: U_k \rightarrow V_k$ such that $\phi_k(U_k \cap M_k) = (\Isotropies \times N_k) \cap V_k$, where $V_k$ is an open neighborhood of $N'_0 \coloneqq \{\id\} \times N_0 \subset \Isotropies \times M$, where $\id = (\id_{\C^2},\id_{\C^3})$. Moreover $\phi_k(U_k \cap N_k) = (\{\id\} \times N_k)\cap V_k$ and $\phi_k$ satisfies the properties given in \autoref{thm:localTrivial}. We define $\phi: U_0 \rightarrow V_0$, $\phi(x) \coloneqq \phi_k(x)$ for $x \in U_0 \cap U_k$, where $k=1,2$ and $V_0 = V_1 \cup V_2$ is an open neighborhood of $N_0'$. Write $\widetilde U \coloneqq U_1 \cap U_2 \cap U_0 \subset \FreeFTwo$ for an open neighborhood of $N_0$. Then we have $\phi\big|_{\widetilde U} = \phi_1\big |_{\widetilde U} =  \phi_2\big |_{\widetilde U}$, which implies that $\phi$ is a real-analytic diffeomorphism. Furthermore, since $\image(\phi_1 |_{\widetilde U \cap M}) = \image(\phi_2 |_{\widetilde U \cap M}) = (\Isotropies \times N_0) \cap \widetilde V$, where $\widetilde V$ is an open neighborhood of $N'_0\subset \Isotropies \times M$, the mapping $\phi$ locally maps $M_0$ real-analytically diffeomorphic to $\Isotropies \times N_0$. Finally the last statement follows from \autoref{thm:localTrivial}, since if $\FreeFTwo$ would be a smooth manifold, then the quotient $\FreeNTwo$ needs to be a smooth manifold, by the smooth version of \autoref{thm:localTrivial}, see also \cite[Theorem 1.11.4]{dk}, which is not the case.
\end{proof}

\begin{proof}[Proof of \autoref{theorem:quotientTop}]
We show that $\Pi: \FreeFTwo \rightarrow \FreeNTwo$ is a surjective, continuous and closed mapping with respect to $\tau_J$. Surjectivity is clear from \autoref{proposition:NormalForm2Nondeg} and \autoref{theorem:ReductionOneParameterFamilies2}. To show continuity of $\Pi$ with respect to $\tau_J$ we let $(H_n)_{n\in \N}$ be a sequence of mappings in $\FreeFTwo$ and $H \in \FreeFTwo$, such that $H_n \rightarrow H$. Assuming w.l.o.g. that $H \in \FreeNTwo$ we need to conclude that $\Pi(H_n) \rightarrow H$. We have $\Pi(H_n) = \phi'_n \circ H_n \circ \phi_n^{-1} \in \FreeNTwo$, where $(\phi_n,\phi'_n) \in \Isotropies$ are the isotropies according to \autoref{proposition:NormalForm2Nondeg}. Assume $ \phi'_n \circ H_n \circ \phi_n^{-1} \rightarrow \widehat H \in \FreeNTwo$, then by \autoref{lemma:SequenceN2Action} we obtain $(\phi_n, \phi_n') \rightarrow (\id_{\C^2},\id_{\C^3})$ and since $H_n \rightarrow H$ we get $\widehat H = H$. We are left by proving the closedness of $\Pi$ with respect to $\tau_J$: Let $C \subset \FreeFTwo$ be a closed subset. We need to show that $\Pi(C) \subset \FreeNTwo$ is a closed subset. To prove this statement we let $H_n \in \Pi(C)$ for $n\in \N$, forming a sequence of mappings in $\FreeNTwo$ such that $H_n \rightarrow H_0$, where $H_0 \in \FreeNTwo$. For the closedness of $\Pi(C)$ we need to conclude that $H_0 \in \Pi(C)$. By \autoref{theorem:ReductionOneParameterFamilies2} we can write $H_n = \MapOneParameterFamilyThree{k_n}{s_n}{\eps}$ and $H_0 = \MapOneParameterFamilyThree{k_0}{s_0}{\eps}$ for $k_n,k_0 \in \{2,3\}$. Note that since $H_n \rightarrow H_0$ in $\FreeNTwo$ we have $s_n \rightarrow s_0$. This implies that for any convergent sequence $G_n \in \Pi^{-1}(H_n)$ the map $G_0 \coloneqq \lim_{n \rightarrow \infty} G_n$ belongs to $\Pi^{-1}(H_0)$. Since $C$ is closed, an arbitrary convergent sequence $F_n \in \Pi^{-1}(H_n) \cap C$  with $F_n \rightarrow F_0$ thus satisfies $F_0 \in \Pi^{-1}(H_0) \cap C$, which implies $H_0 = \Pi(F_0) \in \Pi(C)$. 
\end{proof}

\address{Texas A\&M University at Qatar, PO Box 23874, Doha, Qatar}\\
\email{michael.reiter@qatar.tamu.edu}
\end{document}